\documentclass[12pt, reqno]{amsart}
\usepackage{amsthm}
\usepackage{amsfonts, ulem}
\usepackage{amsmath}
\usepackage{amssymb}
\usepackage{mathrsfs}
\usepackage{bbm}
\usepackage{tikz, hhline}
\usepackage{multicol}
\usepackage{tabularx, multirow}
\usepackage[utf8]{inputenc}
\usetikzlibrary{matrix}
\usepackage{enumitem}
\usepackage{color}
\usepackage[all]{xy}

\everymath{\displaystyle}
\CompileMatrices

\CompileMatrices

\newtheorem{thm}{Theorem}[section]
\newtheorem{lemma}[thm]{Lemma}
\newtheorem{prop}[thm]{Proposition}

\theoremstyle{definition}
\newtheorem{defn}[thm]{Definition}

\newtheorem{const}[thm]{Construction}

\def\ve{\varepsilon}

\def\R{{\mathbb R}}

\def\C{{\mathbb C}}
\def\H{{\mathbb H}}

\def\N{{\mathbb N}}
\def\T{{\mathbb T}}
\def\Z{{\mathbb Z}}

\def\Up{{\Upsilon}}

\def\1{{1}}

\def\so{_{\scriptscriptstyle O}}
\def\su{_{\scriptscriptstyle U}}

\def\pr{^{\scriptscriptstyle \R}}

\def\sc{_{\scriptscriptstyle \C}}
\def\crr{^{\scriptscriptstyle {\it CR}}}
\def\crt{^{\scriptscriptstyle {\it CRT}}}
\def\scrt{_{\scriptscriptstyle {\it CRT}}}
\def\CRT{\mathcal{CRT}}

\def\T{ ^{\rm{Tr}}}
\def\sT{ ^{\sharp \otimes {\rm{Tr}}}}
\def\Ttau{ ^{\rm{Tr} \otimes \tau}}
\def\sTtau{^{\sharp \otimes {\rm{Tr}} \otimes \tau}}
\def\stau{^{\sharp \otimes \tau}}

\def\diag{\rm{diag}}

\newcommand{\fa}{\mathfrak a}

\newcommand{\fS}{\mathfrak S}

\newcommand{\sm}[4]
	{ \left( \begin{smallmatrix} {#1} & {#2} \\ {#3} & {#4} \end{smallmatrix} \right)  }

\def\id{\text {id} \,}

\def\hom{\text {Hom}}

\begin{document}

~

\vspace{-1cm}

\title{The real {$K$}-Theory of the Sphere with the Antipodal Involution} 
\author{Jeffrey L. Boersema}
\begin{abstract}
This is a thorough investigation on the real $K$-theory of the sphere $S^d$ associated with the antipodal involution. We calculate the algebraic structure of real $K$-theory and united $K$-theory for all $d$, we write down explicit unitaries representing the generators of all the non-trivial $K$-theory groups for $d \leq 4$, and we describe a recipe for generating such unitaries for all $d$.
\end{abstract}

\maketitle




\section{Introduction}

Recall, from \cite{blackadarbook} or \cite{parkbook}, that the complex K-groups on the sphere 
$$S^d = \{x \in \R^{d+1} \mid \| x \| = 1\}$$ are given by
$$K^*(S^d) = K_*(C(S^d)) =  \begin{cases} ( \Z \oplus \Z, 0) & \text{$d$ even} \\  ( \Z , \Z) & \text{$d$ odd.}
	\end{cases} $$
where $C(S^d)$ denotes the set of complex-valued continuous functions on the sphere $S^d$. Recently an iterative construction was given in \cite{SBS} to construct unitaries in matrix algebras over $C(S^d)$ that will represent the appropriate generators of the non-zero $K$-theory classes, for all $d$. In this article we will conduct a similar study of the real $K$-theory of the sphere $S^d$ with the antipodal involution $\fa$, using the unitary picture of real $K$-theory for real $C \sp *$-algebras developed in \cite{BL}.

It is a familiar adage that theory of commutative unital $C \sp *$-algebras is equivalent to the theory of compact Hausdorff spaces; and that the theory commutative non-unital $C \sp *$-algebras is equivalent to the theory of locally compact Hausdorff spaces.
Furthermore under this equivalence, the topological $K$-theory corresponds to the operator algebra $K$-theory, via a natural isomorphism 
$${K}^i(X) \cong K_i(C(X))$$
between graded topological $K$-theory groups and the graded operator $K$-theory groups, where $X$ is a compact Hausdorff space. 

Less-familiar perhaps is the adage that the theory of commutative unital real $C \sp *$-algebras is equivalent to the theory of compact Hausdorff spaces with involution; and that the theory commutative non-unital real $C \sp *$-algebras is equivalent to the theory of locally compact Hausdorff spaces with involution. Indeed every real commutative unital $C \sp *$-algebra is isomorphic to 
$$C(X, \tau) = \{ f \in C(X) \mid f(x^\tau) = \overline{f(x)} \}$$
for some compact space with continuous involution $(X, \tau)$; and every real commutative non-unital $C \sp *$-algebra is isomorphic to
$C_0(X, \tau)$ for some locally compact space $X$ with involution (see Theorem~2.7.2 of \cite{Li2003}).
Furthermore, there is a natural isomorphism
$$KR^i(X, \tau) \cong KO_i(C(X, \tau)) \; .$$ 
Consider $i = 0$. On the left side $KR^0(X, \tau)$ is the invariant introduced by Atiyah in \cite{atiyah66}, the Grothiendieck group of the semigroup $V(X, \tau)$ of isomorphism classes of pairs $(E, \sigma)$, where $E$ is a complex vector bundle with projection map $\pi \colon E \rightarrow X$, and where $\sigma$ is a conjugate linear involution on $E$ that intertwines with $\tau$.
On the right side, $KO_0(C(X, \tau))$ is the operator algebra $K$-theory for the real $C \sp *$-algebra $C(X, \tau)$, defined as the Grothiendeck group of the semigroup of equivalence classes of projections in $M_\infty(C(X, \tau))$.

On the sphere $S^d \subset \R^{d+1}$,
there exists a family of nonequivalent involutions $ \tau^{a,b}$ 
where $a+b = d+1$ given by
$$\tau^{a,b} \colon (x_1, \dots, x_a, x_{a+1}, \dots, x_{a+b}) \mapsto (x_1, \dots, x_a, -x_{a+1}, \dots, -x_{a+b}) \; $$
for $x = (x_1, \dots, x_{d+1}) \in S^d$. 

If $a \neq 0$ then the involution $\tau^{a,b}$ on $S^d$ has fixed points, and by performing a sterographic projection on a selected fixed point of $S^d$ we find that $C(S^d, \tau^{a,b})$ is isomorphic to the unitization of $C_0(\R^{d}, \tau^{a-1,b})$. In that case, we have an isomorphism $C_0(\R^d, \tau^{a-1,b}) \cong S^{a-b-1} \R$. Recall that for a real $C \sp *$-algebra $A$, the suspensions and desuspensions are defined by
\begin{align*}
SA &= C_0(\R, A) = \{ f \in C_0(\R, A_\C) \mid f(x) = \overline{f(x)} \} \\
S^{-1} A&= \{ f \in C_0(\R, A_\C) \mid f(-x) = \overline{f(x)} \}  \; . \\
\end{align*}
Then $S^n A$ is defined for any $n \in \Z$ by iterating these operations appropriately.
They satisfy
$$KO_*(S^n A) = \Sigma^{n} KO_*(A) $$
for any $n \in \Z$ by Proposition~1.20 of \cite{Boer2002}. 
Here the suspension $\Sigma^n M$ of a $\Z$-graded module $M$ is defined by $(\Sigma^n M)_k = M_{n+k}$.
Thus 
\begin{equation} \label{KOformula} 
KO_*(C(S^d, \tau^{a,b})) = KO_*(\R) \oplus \Sigma^{a-b-1} KO_*(\R) \; .\end{equation}

If $a = 0$ then $\tau^{0,d+1}$ is the antipodal involution $\fa = \tau^{0,d+1}$, which has no fixed points. This space $(S^d, \fa)$ with the antipodal involution will be the focus of this article.
Let
\[ \fS^d = C(S^d , \fa) = \{ f \in C(S^d, \C) \mid f(-x) = \overline{f(x)} \} \]
denote the associated real $C \sp *$-algebra. Since $\fa$ has no fixed points on the sphere, we cannot carry out the $K$-theory calculation for $\tau^{a,b}$ when $a \neq 0$. However, it turns out that the same formula (\ref{KOformula}) still holds for 
 $\fa = \tau^{a,b} = \tau^{0,d+1}$ as long as $d \geq 2$. This is Theorem~\ref{fS^d}.
 
The cases for $d \leq 1$ are exceptional. We have $\fS^0 \cong \C$ and $\fS^1 \cong T$ where $T$ is the real $C \sp *$-algebra associated with self-conjugate $K$-theory from \cite{Boer2002}. The $K$-theory of both of these real $C \sp *$-algebras are known, given by
\begin{align*}
KO_*(\fS^0) &= (\Z, 0, \Z, 0, \Z, 0, \Z, 0) \\
KO_*(\fS^1) &= (\Z, \Z_2, 0, \Z, \Z, \Z_2, 0, \Z)
 \end{align*}
in degrees $0, 1, \dots, 7$. They clearly do not follow the pattern of Formula~\ref{KOformula}.

We are not aware of any other study of the real $C \sp *$-algebras $\fS^d$ and their $K$-theory $KO_*(\fS^d)$, nor of $KR^*(S^d, \fa)$ for $d \geq 2$.
In this article we will present the calculation of the real $K$-theory of $\fS^d$ for all $d$ (in fact, we will compute the united $K$-theory, which includes the real, complex, and self-conjugate $K$-theory).

Beyond the calculation of the abstract algebraic structure of $KO_*(\fS^d)$, we will identify specific unitary classes representing the non-trivial $K$-theory elements for all $i$ in the low-dimensional cases (for $d \leq 4$); and following that we will describe a recipe to produce unitary elements which represent the key generator classes of $KO_{d+2}(\fS^d)$ in any dimension $d \geq 3$. To explain what we mean by these ``unitary generators", recall that
for a real unital $C \sp *$-algebra $A$, the group $KO_0(A)$ is generated by projection classes $[p]$ where $p \in M_n(A)$. By considering $u = 2p-1$ we can instead formulate $KO_0(A)$ in terms of homotopy classes of self-conjugate unitaries $[u]$ where $u \in M_n(A)$.
The group $KO_1(A)$ is represented by homotopy classes of (not necessarily self-conjugate) unitaries $[u]$ where $u \in M_n(A)$ (these statements hold in the real case exactly the same as in the complex case -- see \cite{rordambookblue}).  Thus elements of $KO_0(A)$ and $KO_1(A)$ can both be represented using equivalence classes of unitaries in matrix algebras over $A$.

For $i \neq 0,1$ in the real case, we often tend to work with $KO_i(A)$ abstractly either in terms of suspension, writing $KO_i(A) = KO_0(S^d A)$ (such as in Definition~1.4.4 of \cite{schroderbook}); or in terms of graded Clifford algebras, writing $KO_i(A) = KO_0(A \otimes \mathcal{C}_{i, 0})$ (such as in Theorem 2.4.10 of \cite{schroderbook}). However, these abstract approaches, while useful for studying the functorial properties of real $K$-theory and for calculating the algebraic structure in examples, make it difficult to represent particular elements of $KO_i(A)$ in a concrete and meaningful way. The papers \cite{BL} and \cite{Boer2020} develop a framework for describing elements in all of the groups $KO_i(A)$, as well as $KU_i(A)$, in very concrete terms, using unitaries in $M_n(A \otimes \C)$ that satisfy certain symmetries. The unitary picture of $KO_i(A)$ from \cite{BL} is summarized in Table~\ref{unitaryTable}, which we describe in more detail now.

\newcommand\TT{\rule{0pt}{2.6ex}} 
\newcommand\BB{\rule[-1.2ex]{0pt}{0pt}} 
\begin{table}[h] 
\caption{$KO_*(A)$ and $KU_*(A)$ via unitaries.} \label{unitaryTable}
\begin{center}
\begin{tabular}{c|c|c|c|c|} 
\hhline{~----}
& K-group \TT \BB & $n_i$ & $\mathscr{S}_i$ & $I^{(i)}$ \\ \hhline{=====}
\multicolumn{1}{|c|}{\multirow{2}{*}{complex}}  & $KU_0(A)$ \TT \BB & 2 & $u = u^*$ & $\sm{\1}{0}{0}{-\1}$ \\ \hhline{~----}
\multicolumn{1}{|c|}{} & $KU_1(A)$ \TT \BB & 1 & -- & $\1$ \\ \hhline{=====}
\multicolumn{1}{|c|}{} & $KO_0(A)$ \TT \BB & 2 &  $u = u^*$, $u\Ttau = u $ & $\sm{\1}{0}{0}{-\1}$  \\ \hhline{~----}
\multicolumn{1}{|c|}{\multirow{8}{*}{real}} & $KO_1(A)$ \TT \BB & 1 &  $u\Ttau = u^* $ & $\1$ \\ \hhline{~----}
\multicolumn{1}{|c|}{} & $KO_2(A)$ \TT \BB & 2 &  $u = u^*$, $u\Ttau = -u$ & $\sm{0}{i \cdot \1}{-i \cdot \1}{0}$ \\\hhline{~----}
\multicolumn{1}{|c|}{}  & $KO_{3}(A)$ \TT \BB & 2 & $u\sTtau = u$ & $ \1_2 $ \\
\multicolumn{1}{|c|}{} & \text{or} & 2 &  $u\Ttau = -u$ & { $\sm{0}{\1}{-\1}{0}$ }\\
 \hhline{~----}
\multicolumn{1}{|c|}{} & $KO_4(A)$ \TT \BB & 4 &  $u = u^*$, $u\sTtau = u$ & ${\diag}(\1_2,-\1_2)$ \\ \hhline{~----}
\multicolumn{1}{|c|}{} & $KO_5(A)$ \TT \BB & 2 &  $u\sTtau = u^*$ & $\1_2$ \\ \hhline{~----} 
\multicolumn{1}{|c|}{} & $KO_6(A)$ \TT \BB & 2 &  $u = u^*$, $u\sTtau = -u$ & $\sm{0}{i \cdot \1}{-i \cdot \1}{0}$ \\ \hhline{~----}
\multicolumn{1}{|c|}{} & $KO_{7}(A)$ \TT \BB & 1 & $u\Ttau = u $ & $\1$ \\ 
\multicolumn{1}{|c|}{} & \text{or} & { 2}  & {  $u\sTtau = -u$} & { $\sm{0}{\1}{-\1}{0}$ }\\
\hhline{=====}
\end{tabular}
\end{center}
\end{table}

For any real $C \sp *$-algebra $A$, let $\tau$ be the corresponding antimultiplicative involution on $A\sc = A \otimes \C$, given by
$\tau \colon a + ib \mapsto a^* + ib^*$ where $a,b \in A$. The real $C \sp *$-algebra $A$ can be recovered from the pair $(A\sc, \tau)$ as the elements in $A\sc$ that satisfy $a^\tau = a^*$. The involution $\tau$ extends to involutions $\rm{Tr} \otimes \tau$ on $M_n(\C) \otimes A\sc = M_n(A\sc)$ and involutions $\sharp \otimes \rm{Tr} \otimes \tau$ on $M_2(\C) \otimes M_n(\C) \otimes A\sc = M_{2n}(A\sc)$.
The involution $\sharp$ on $M_2(\C)$ is the involution associated with the algebra of quaternions $\H \subset M_2(\C)$ and is given by
$$\begin{bmatrix} a & b \\ c & d \end{bmatrix}^\sharp = \begin{bmatrix} d & -b \\ -c & a \end{bmatrix} \; .$$
When we are making explicit calculations with matrices, we always choose an isomorphism $M_2(\C) \otimes M_n(\C) \cong M_{2n}(\C)$ such that
	$$a^{\sharp \otimes \rm{Tr}}
		= \begin{bmatrix} a_{1,1} & a_{1,2} & \cdots & a_{1,n} \\ a_{2, 1} & a_{2,2} & \cdots & a_{2,n} \\
			\vdots & \vdots & \ddots & \vdots \\ a_{n,1} & a_{n,2} & \cdots & a_{n,n}  \end{bmatrix}^{\sharp \otimes \rm{Tr}} 
	 	= 	\begin{bmatrix} a_{1,1}^{\sharp} & a_{2,1}^\sharp & \cdots & a_{n,1}^\sharp \\
			 \\ a_{1,2}^\sharp & a_{2,2}^\sharp & \cdots & a_{n,2}^\sharp \\
			\vdots & \vdots & \ddots & \vdots \\ a_{1,n}^\sharp & a_{2,n}^\sharp & \cdots & a_{n,n}^\sharp \end{bmatrix}
		$$
where each $a_{i,j} \in M_2(\C)$.
		
For a unital real $C \sp *$-algebra, the elements of $KO_i(A)$ are represented by unitaries in the complexification $M_{n}(A\sc)$ that satisfy the relation $\mathscr{S}_i$ given in the table (and where $n$ is a multiple of $n_i$). The special unitary $I^{(i)}$ in the table represents a neutral element, which satisfies $[ I^{(i)} ] = 0$ in $KO_i(A)$.  The identification $[u] = [\diag(u, I^{(i)})] \in KO_i(A)$ shows how to identify a unitary in $M_n(A\sc)$ with a unitary in $M_{n+n_i}(A\sc)$ representing the same $KO$-class. 

There are two lines in the table for each of $KO_{3}(A)$ and $KO_7(A)$. The first line represents the picture of $KO_i(A)$ developed in \cite{BL} and also used in \cite{Boer2020}. The second line represents a different variation of the unitary picture of $KO_i(A)$ that we will find more convenient later in this present work.
This variation is described in detail in the appendix in Section~\ref{newunitaries}. The isomorphism between the two unitary pictures of $KO_3$ and $KO_7$ includes a constructive description of how to convert a unitary in one picture to a unitary in the the other picture. This alternate picture of unitaries-based $KO$-theory will sometimes be used in Section~\ref{Section:S^n unitaries} to describe the generators $g_d$ in $KO_{d+2}(\fS^d)$ for all $d$.

Our results in Section~\ref{Section:S^n unitaries}  extend to the real case, and heavily rely on, the results of \cite{SBS}. 
There, Schulz-Baldes and Stoiber describe a recipe for representing the non-trivial elements of the complex $K$-theory $K_i(C(S^d))$ for all $d \geq 1$ in terms of self-adjoint unitaries (for $i = 0$) and unitaries (for $i = 1$).
 Briefly, for a positive odd integer $d$, first write down a set of Clifford generators $\Gamma_{1, d}, \dots \Gamma_{d,d}$, which are matrices in $M_{n}(\C)$ where
 $n = 2^{(d-1)/2}$ that satisfy 
 $$\Gamma_{i, d} \Gamma_{j,d} +  \Gamma_{j, d} \Gamma_{i,d}  = 2 \delta_{i,j} \; .$$
  Then we obtain a self-adjoint unitary $Q_{d-1} \in C(S^{d-1}, M_n(\C))$ and a unitary $U_d \in C(S^{d}, M_n(\C))$ by
\begin{align} \label{QandU}
	Q_{d-1}(x) &= \sum_{i=1}^{d} x_i \Gamma_{i,d} \\
	 \text{and} \quad 
	U_d(x) &= \sum_{i=1}^d x_i \Gamma_{i,d}  + x_{d+1} i \,  I   \; .  \nonumber
\end{align}
Then it is proven in Proposition~1 of \cite{SBS} that the elements
$$[Q_{d-1}] \in \widetilde{K}_0(C(S^{d-1}, \C)) = \Z \quad \text{and} \quad
[U_d] \in \widetilde{K}_1(C(S^d, \C) ) = {K}_1(C(S^d, \C) ) = \Z$$ represent generators of the respective $K$-theory groups.

We will use a variation the construction of \cite{SBS} to obtain unitaries that represent the key generator $g_d \in KO_{d+2}(\fS^d) = \Z$ for all $d$. Of course, the construction will vary depending on the value of $d$ modulo $8$, because the symmetry that the unitary must satisfy varies depending on $d$. Our recipe will start with the Clifford generators $\Gamma_{1, d}, \dots \Gamma_{d,d}$ and then modify them to present a new set of Clifford generators
$\Up_{1, d}, \dots \Up_{d,d}$ that has the symmetry properties needed to produce the correct unitary classed $Q_{d-1}\pr$ and $U_d\pr$
using the same formal expression as in (\ref{QandU}) in the complex case. When not considering extra symmetry structures associated with real $K$-theory, the generators $\Up_{1, d}, \dots \Up_{d,d}$ are equivalent to $\Gamma_{1, d}, \dots \Gamma_{d,d}$ by a unitary equivalence. So the complex $K$-theory elements $[U]$ and $[Q]$ are the same when using $\Up_{i,d}$ as when using $\Gamma_{i,d}$ in the formulas in (\ref{QandU}).

Our results are related to results in \cite{JosephMeyer} in which some generators of the real $K$-theory of spheres with involution are also identified explicitly. That work utilizes the van Daele picture of $K$-theory, which is not precisely the same as the unitary picture of elements from \cite{BL} that we are using here. In the van Daele picture, $K$-theory classes are represented by odd self-adjoint unitaries in the graded algebra $\mathcal{K} \otimes A \otimes \mathcal{C}\ell_{p,q}$ where $\mathcal{C}\ell_{p,q}$ is a graded Clifford algebra. By contrast, in the unitary picture the we use here, developed in \cite{BL} and summarized above, all $K$-theory classes are represented by unitaries (self-adjoint unitaries in the even cases and not-necessarily-self-adjoint unitaries in the odd cases) satisfying certain symmetry relations in matrix algebras over the complexification of $A$. This avoids the need to move the category of graded $C \sp *$-algebras and also avoids the need to consider Clifford algebras a priori (though we will certainly make extensive use of Clifford generators in our constructions). Furthermnore, in \cite{JosephMeyer}, they specifically restrict the involution $\tau^{a,b}$ to the case $a \geq 1$ on sphere $S^d$, whereas we specifically focus on the case $a = 0$. (We ourselves will consider the cases for $a \geq 1$ in a separate forthcoming paper.)

The unitary picture of $KO_i(A)$ is largely motivated by its utility in physics. Formulas used to detect topological markers associated with band gaps can be developed and best-understood in a $K$-theory framework. Furthermore, the symmetries of the underlying model determine what kind of topological markers can be present. This involves potential spatial symmetries as well as time-symmetry. Mathematically this is best-understood in terms of the real $K$-theory of the appropriate space-with-involution parametrizing the model. Although the correct formulas can often be obtained by experienced guess-work, the $K$-theory classes of the appropriate $C \sp *$-algebra give them a mathematical account. This perspective has motivated the work of \cite{SBS} in the complex case, as well as the work in \cite{BL} and \cite{Boer2020} discussed above, and the work in this paper.

Such formulas are developed into algorithms that are used to detect topological markers in physical models such as in \cite{bradlyn},  \cite{Li...2025}, and \cite{Loring2015}.  These methods are used to study topological insulators and more recently these methods have been proposed to classify topology in photonic systems as in \cite{CL-2022}, \cite{CL2024}, \cite{CLSB-2024}, \cite{DLC-2023}.
We wish to be able to address physical models in a variety of possible dimensions, with a variety of symmetry structures, which may or may not contain time-reversal symmetry, and which may or may not contain particle-hole symmetry. The formulas developed especially in Section 3 are   topological markers that could be used as topological markers in low-dimension models. Thus this article extends the range of such models that can be addressed when the key symmetries correspond with the antipodal map of the sphere.

For example, in \cite{Li...2025} a $K$-theoretic topological marker was developed for a physical model involving a $C_2 \mathcal{T}$ symmetry.  Here $C_2$ is a rotational symmetry and $\mathcal{T}$ is a time-reversal symmetry and the composition $C_2 \mathcal{T}$ commutes with the Hamiltonian $H$ and anit-commutes with two position operators $X$ and $Y$. Thus the system is associated with a sphere with a $\tau^{2,1}$ involution. One can easily imagine scenarios which would be modeled by the $\tau^{0,2}$ involution on $S^1$ or the $\tau^{0,3}$ involution on $S^2$, and these would fall into the situation discussed in this paper. 

The antipodal map is a relatively special case in the space of involutions on the sphere. Our focus on that case in this paper was motivated primarily by the expediency of restricting the problem to a manageable size. This special case seemed to us to have special mathematical interest and is likely to be useful for physical applications. However, we plan to continue this project and in a forthcoming paper, we will extend our formulas to include unitaries which represent $KO_*$ classes for spheres equipped with all possible involutions $\tau^{a,b}$, besides just the antipodal map. There will also be value in expanding this work to include the various tori with involutions, as was done for van Daele $K$-theory in \cite{JosephMeyer}.

\subsection{Acknowledgements}

This work was supported in part by the Laboratory Directed Research and Development program at Sandia National Laboratories.
Also, this work was performed, in part, at the Center for Integrated Nanotechnologies, an Office of Science User Facility operated for the U.S. Department of Energy (DOE) Office of Science by Los Alamos National Laboratory (Contract 89233218CNA000001) and Sandia National Laboratories (Contract DE-NA-0003525).

\section{Abstract $K$-theory and $K$-homology groups}

In this section, we will compute the algebraic structure of $KO_*(\fS^d)$ for all $d$. But first, we will review the fundamentals of united $K$-theory $K\crt(A)$ and $K\crr(A)$ for a real $C \sp *$-algebra, introduced in \cite{Boer2002} and \cite{bousfield90}.
These two variations of united $K$-theory are often used to augment the real $K$-theory $KO_*(A)$ for a variety of reasons. One reason is that united $K$-theory (in either version) is a more complete invariant: $KO_*(A)$ by itself does not classify real $C \sp *$-algebras up to $KK$-equivalence, but $K\crr(A)$ and $K\crt(A)$ do so (for real $C \sp *$-algebras in the real bootstrap category) by \cite{Boer2004}. In the category of real Kirchberg algebras in the bootstrap category, united $K$-theory even classifies real $C \sp *$-algebras up to isomorphism.
Furthermore, when calculating, it is often very convenient to consider $KO_*(A)$ as part of $K\crr(A)$ or $K\crt(A)$, since the additional structure of united $K$-theory, in the category of $\mathcal{CRT}$-modules, can be used to obtain information of $KO_*(A)$.

For a real $C\sp*$-algebra $A$, we define united $K$-theory (with two variations) by
\begin{align*}
	K\crr(A) &= \{ KO_*(A), KU_*(A) \} \\
	K\crt(A) &= \{ KO_*(A), KU_*(A), KT_*(A) \}
\end{align*}
where $KO_*(A)$ is the usual period-8 real $K$-theory of $A$ and $KU_*(A) = K_*(\C \otimes A)$ is the $K$-theory of the complexification of $A$. 
The ``self-conjugate" $K$-theory, which is included the full variation of united $K$-theory, is defined by $KT_*(A) = K_*(T \otimes A)$ where $T$ is the algebra defined in \cite{Boer2002} and which turns out to be identical to what we are calling $\fS^1$ in this article.

Recall that $KO_*(A)$ has the structure of a graded module over the ring $KO_*(\R)$ where the groups of this ring are given by
\[ KO_*(\R) = (\Z , ~\Z_2 ,~ \Z_2 ,~ 0  ,~ \Z  ,~ 0  ,~ 0  ,~ 0)  \]
in degrees 0 through 7. 

In particular, 
multiplication by the non-trivial element $\eta$ of $KO_1(\R) \cong \Z_2$
 induces a natural transformation
$\eta \colon KO_i(A) \rightarrow KO_{i+1}(A) .$
We note that $\eta$ satisfies the relations $2 \eta = 0$ and $\eta^3 = 0$, both as an element of the ring $KO_*(\R)$ and as a natural transformation. There is also a non-trivial element $\xi \in KO_4(\R)$, and corresponding natural transformation
$\xi \colon KO_i(A) \rightarrow KO_{i+4}(A)$. It satisfies $\xi^2 = 4 \beta\so$ where $\beta\so$ is the real Bott periodicity isomorphism of degree 8.

Complex $K$-theory $KU_*(A)$ has the structure of a module over $KU_*(\R) = K_*(\C)$, but the only natural transformation which arises from this structure is the degree 2 Bott periodicity map $\beta\su$. There is, however, a natural transformation $\psi \colon KU_*(A) \rightarrow KU_*(A)$ that arises from the conjugation map $\psi \colon \C \otimes A \rightarrow \C \otimes A$ defined by $a + ib \mapsto a - ib$.

In addition, there are natural transformations
\begin{align*}
c \colon & KO_*(A) \rightarrow KU_*(A) \\
r \colon & KU_*(A) \rightarrow KO_*(A) 
\end{align*}
which are induced by the natural inclusion maps $\R \hookrightarrow \C$ and $\C \hookrightarrow M_2(\R)$, respectively. 

Taken together, these natural transformations satisfy the following set of relations:
\begin{align*}
\label{eq:natural-transformations}
rc &= 2  & cr &= 1 + \psi &2 \eta &= 0  \notag \\
\eta r &= 0 & c \eta&= 0 &  \eta^3 &= 0    \notag \\
r \psi &= r & \psi^2 &= \id &  \xi^2 &= 4 \beta\so \\ 
\psi c &= c & \psi \beta\su &= -\beta\su \psi & \xi &= r \beta\su^2 c \notag 
\end{align*} 
The natural transformations also combine to form a long exact sequence
\begin{equation}
\label{eq:CR-LES}
\dots \rightarrow
KO_{i}( A) \xrightarrow{\eta} 
KO_{i+1}( A) \xrightarrow{c} 
KU_{i+1}( A) \xrightarrow{r \beta\su^{-1}} 
KO_{i-1}(A) \rightarrow
\cdots  \; .
\end{equation}

There is even more structure to $K\crt(A) = \{ KO_*(A), KU_*(A), KT_*(A) \}$, including several natural transformations that involve the self-conjugate $K$-theory and two additional exact sequences that involve self-conjugate $K$-theory. We will not review this in detail, but refer the reader to \cite{Boer2002} and \cite{bousfield90}.

For us, the two important facts about united $K$-theory are the following. First, if $A$ and $B$ are real $C \sp *$-algebras such that $A \otimes \C$ and $B \otimes \C$ are in the bootstrap category $\mathcal{N}$ defined in \cite{schochet87}, then $A$ and $B$ are $KK$-equivalent if and only if $K\crt(A) \cong K\crt(B)$, if and only if $K\crr(A) \cong K\crr(B)$. This statement follows from the Universal Coefficient Theorem for real $C \sp *$-algebras in \cite{Boer2004} together with the algebraic work of \cite{hewitt} on abstract $\mathcal{CRT}$-modules.

Secondly, there is a very nice abstract characterization of free $\mathcal{CRT}$-modules from \cite{bousfield90}, which was crucial in the proving the K\"unneth Formula for real $C \sp *$-algebras in \cite{Boer2002}, and will also be useful here in our computations. 
Namely, for any real $C \sp *$-algebra, $K\crt(A)$ is free (as a $\mathcal{CRT}$-module) if and only if the complex part $KU_*(A)$ is free (as a graded abelian group), if and only if $K\crt(A)$ is isomorphic to a direct sum of suspensions of $K\crt(\R)$, $K\crt(\C)$, and $K\crt(T)$. 

Since $KU_*(\fS^d) = K_*(C(S^d))$ is free (as a graded abelian group with period 2), it follows that $K\crt(\fS^d)$ is a free $\mathcal{CRT}$-module for all $d$. So
$K\crt(\fS^d)$ can be expressed as a direct sum of $K\crt(\R)$,  $K\crt(\C)$, and $K\crt(T)$ and their suspensions. Here we will focus on the smaller version of united $K$-theory $K\crr(\fS^d)$, for which a similar statement is true, namely that $K\crr(\fS^d)$ can be expressed as a direct sum of $K\crr(\R)$,  $K\crr(\C)$, and $K\crr(T)$ and their suspensions. For reference, we display the groups of these $\mathcal{CR}$-modules in Table~\ref{crrRCT}.
\begin{table}[h]  \label{crrRCT}
\caption{The free $\mathcal{CR}$-modules}
$$K\crr(\R)$$
\[ \begin{array}{|c|c|c|c|c|c|c|c|c|c|}  
\hline \hline  
n & ~0~ & ~1~ & ~2~ & ~3~ & ~4~ & ~5~ & ~6~ & ~7~  \\
\hline  \hline
KO_n
& \Z  & \Z_2  & \Z_2  & 0 
	& \Z  & 0 & 0 & 0  \\
\hline  
KU_n 
& \Z  & 0 & \Z  & 0 & \Z   & 0 
	& \Z  & 0   \\
\hline \hline
\end{array} \]

\vspace{1cm}

$$K\crr(\C)$$
\[ \begin{array}{|c|c|c|c|c|c|c|c|c|c|}  
\hline  \hline 
n & \makebox[1cm][c]{0} & \makebox[1cm][c]{1} & 
\makebox[1cm][c]{2} & \makebox[1cm][c]{3} 
& \makebox[1cm][c]{4} & \makebox[1cm][c]{5} 
& \makebox[1cm][c]{6} & \makebox[1cm][c]{7}  \\
\hline  \hline
KO_n 
& \Z & 0 & \Z & 0 & \Z & 0 & \Z & 0   \\
\hline  
KU_n 
& \Z \oplus \Z & 0 & \Z \oplus \Z & 0 & \Z \oplus \Z & 0 
	& \Z \oplus \Z  & 0   \\
\hline \hline
\end{array} \]

\vspace{1cm}

$$K\crr(T)$$
\[ \begin{array}{|c|c|c|c|c|c|c|c|c|c|}  
\hline  \hline 
n & \makebox[1cm][c]{0} & \makebox[1cm][c]{1} & 
\makebox[1cm][c]{2} & \makebox[1cm][c]{3} 
& \makebox[1cm][c]{4} & \makebox[1cm][c]{5} 
& \makebox[1cm][c]{6} & \makebox[1cm][c]{7}  \\
\hline  \hline
KO_n 
& \Z & \Z_2& 0 & \Z & \Z & \Z_2& 0 & \Z  \\
\hline  
KU_n 
& \Z & \Z  & \Z & \Z & \Z & \Z & \Z & \Z \\
\hline \hline
\end{array} \]
\end{table}

We now turn to the computation of $K\crr(\fS^d)$. We already know $\fS^0 \cong \C$ and $\fS^1 \cong T$ where $T$ is the ``self-conjugate" circle algebra, and for both of these the $K$-theory is displayed in Table~\ref{crrRCT}. The structures of $K\crr(\fS^d)$ for $d \geq 2$ are given in Theorem~\ref{fS^d} below, and shown explicitly in Table~\ref{lowdim} below for $d = 2,3,4$. 

\begin{thm} \label{fS^d} 
For $d \geq 2$ we have
\begin{align*} K\crt( \fS^d) &\cong
	 K\crt(\R) \oplus K\crt(S ^{-d-2} \R) \\
	 &\cong  K\crt(\R) \oplus \Sigma^{-d-2} K\crt(\R) \; . \end{align*}
The first summand is the injective image of $(\ve_d)_*$ where $\ve_d \colon \R \rightarrow \fS^d$ is the unital inclusion.
\end{thm}

In part, this result states that $K\crt(\fS^d)$ is a free $\mathcal{CRT}$-module with free generators in the real part in degrees 0 and $d+2$.
The first summand is generated (as a $\mathcal{CRT}$-module) by the class of the unit projection $[1] \in KO_0(\fS^d)$ and the second summand is generated by an element in $KO_{d+2}(\fS^d)$, though there may not be a canonical choice of such an generator.

\begin{table} 
\caption{$K\crr(\fS^d)$ for $2 \leq d \leq 4$}  \label{lowdim}

%
%
%
%
%
%

$$K\crr(\fS^2)$$
$$\begin{array}{|c|c|c|c|c|c|c|c|c|}  
\hline \hline  
n & \makebox[1cm][c]{0} & \makebox[1cm][c]{1} & 
\makebox[1cm][c]{2} & \makebox[1cm][c]{3} 
& \makebox[1cm][c]{4} & \makebox[1cm][c]{5} 
& \makebox[1cm][c]{6} & \makebox[1cm][c]{7}  \\

\hline  \hline
KO_n(\fS^2)
& \Z^2  & \Z_2  & \Z_2 & 0 
& \Z^2  & \Z_2  & \Z_2 & 0  \\
\hline  
KU_n(\fS^2)
& \Z^2 & 0 & \Z^2 & 0 & \Z^2 & 0 & \Z^2 & 0  \\

\hline \hline
\end{array}$$

\vspace{1cm}

$$K\crr(\fS^3)$$
$$\begin{array}{|c|c|c|c|c|c|c|c|c|}  
\hline \hline  
n & \makebox[1cm][c]{0} & \makebox[1cm][c]{1} & 
\makebox[1cm][c]{2} & \makebox[1cm][c]{3} 
& \makebox[1cm][c]{4} & \makebox[1cm][c]{5} 
& \makebox[1cm][c]{6} & \makebox[1cm][c]{7}  \\

\hline  \hline
KO_n(\fS^3)
& \Z & \Z_2 \oplus \Z & \Z_2 & 0 
& \Z & \Z & \Z_2 & \Z_2  \\
\hline  
KU_n(\fS^3)
& \Z & \Z & \Z & \Z& \Z & \Z& \Z & \Z  \\

\hline \hline
\end{array}$$

\vspace{1cm}

$$K\crr(\fS^4)$$
$$\begin{array}{|c|c|c|c|c|c|c|c|c|}  
\hline \hline  
n & \makebox[1cm][c]{0} & \makebox[1cm][c]{1} & 
\makebox[1cm][c]{2} & \makebox[1cm][c]{3} 
& \makebox[1cm][c]{4} & \makebox[1cm][c]{5} 
& \makebox[1cm][c]{6} & \makebox[1cm][c]{7}  \\

\hline  \hline
KO_n(\fS^4)
& \Z \oplus \Z_2 & \Z_2 & \Z_2 \oplus \Z & 0 
& \Z & 0 & \Z & \Z_2  \\
\hline  
KU_n(\fS^4)
& \Z^2 & 0 & \Z^2 & 0 & \Z^2 & 0 & \Z^2 & 0  \\

\hline \hline
\end{array}$$
\end{table}

\begin{proof}
As discussed above, since $K\crt( \fS^d)$ is a free $\CRT$-module, $K\crt(\fS^d)$ must be isomorphic to a direct sum of suspensions of the free $\CRT$-modules $K\crt(\R), K\crt(\C),$ and $K\crt(T)$. Furthermore the complex part of $K\crt(\fS^d)$ must satisfy
$$KU_*( \fS^d) = KU_*(C(S^d, \C)) = \begin{cases} (\Z^2, 0) & \text{if $d$ is even } \\ (\Z, \Z) & \text{if $d$ is odd,} \end{cases} $$
and when we compare this to the complex part of the $\CRT$-modules in Table~\ref{crrRCT}, there are few possibilities for $K\crt(\fS^d)$.


In particular, in the case that $d$ is even,
the only possibilities that are consistent the complex part of $K\crt(\fS^d)$ are
$K\crt( \fS^2) \cong K\crt(\C)$ or $K\crt(\fS^2) = \Sigma^a K\crt(\R) \oplus \Sigma^b K\crt(\R)$ where $a,b$ are both even.
In the case that $d$ is odd,
the only possibilities that are consistent the complex part of $K\crt(\fS^d)$ are
$K\crt( \fS^2) \cong K\crt(T)$ or $K\crt(\fS^2) = \Sigma^a K\crt(\R) \oplus \Sigma^b K\crt(\R)$ where $a$ is even and $b$ is odd.

Let $d = 2$ and consider the map $\pi \colon \fS^2 \rightarrow \fS^1$ given by evaluation on the equator.
The kernel $\pi$ is the ideal of functions on the sphere that vanish on the equator and satisfy $f(-x) = -f(x)$. Since such functions are determined by their values on the top half of the sphere, we have $\ker \pi \cong S^2 \C$. Thus we have the short exact sequence
$$0 \rightarrow S^2 \C \xrightarrow{i} \fS^2 \xrightarrow{\pi} \fS^1 \rightarrow 0 \;. $$
Using the fact that $KO_*(S^2 \C) = (\Z, 0)$ a part of the resulting long exact sequence on real $K$-theory is
$$KO_1(S^2 \C) \rightarrow KO_1( \fS^2) \xrightarrow{\pi_*} KO_1( \fS^1) \xrightarrow{\partial_1} KO_0(S^2 \C) \xrightarrow{i_*} 
	KO_0( \fS^2) \xrightarrow{\pi_*} KO_0( \fS^1) \rightarrow KO_{-1}(S^2\C) $$
or 
$$0 \rightarrow KO_1( \fS^2) \xrightarrow{\pi_*} \Z_2 \xrightarrow{\partial_1} \Z \xrightarrow{i_*} 
	KO_0( \fS^2) \xrightarrow{\pi_*} \Z \rightarrow 0 \; .$$
It follows from this that $\partial_1 = 0$ and $KO_0(\fS^2) = \Z^2$ and $KO_1(\fS^2) = \Z_2$. An identical analysis, at the portion of the long exact sequence starting with $KO_5(S^2 \C) = 0$, yields $KO_4(\fS^2) = \Z^2$ and $KO_5(\fS^2) = \Z_2$. From this information, in light of the possibilities discussed in the previous paragraph, it follows that $K\crt(\fS^2) = K\crt(\R) \oplus \Sigma^4 K\crt(\R)$.

Now, observe that $\ve_1 = \pi_{\rm } \ve_2$. We know from Section~2.1 of \cite{Boer2002} that $\ve_1$ induces an isomorphism  $KO_i(\R) \rightarrow KO_i(\fS^1)$ for $i = 0,1$. It follows from this that $\ve_2$ induces an injective map $KO_i(\R) \rightarrow KO_i(\fS^2)$ for $i = 0,1$ which then implies (given the structure of $K\crt(\fS^2)$ identified in the previous paragraph) that $(\ve_2)_*$ is injective, carrying $K\crt(\R)$ isomorphically onto the first shown summand of $K\crt(\fS^2)$. This proves the theorem for $d= 2$.

Now for $d > 2$ we use a generalized ``equator" evaluation map $\pi_{\rm{}} \colon \fS^d \rightarrow \fS^2$ and the identity $\ve_2  = \pi \ve_d$. Since $(\ve_2)_*$ is injective, so is $(\ve_d)_*$, carrying $K\crt(\R)$ to a submodule of $K\crt(\fS^d)$. Again, given the limited possibilities for the structure of $K\crt(\fS^d)$ discussed in the second paragraph of this proof, it follows that there is a decomposition of $K\crt(\fS^d)$ as
$$K\crt(\fS^d) \cong K\crt(\R) \oplus \Sigma^b K\crt(\R) \; $$
where the first summand is the image of $(\ve_d)_*$ and where $b$ has the same parity as $d$.
We call the second summand $\widetilde{K}\crt(\fS^d)$, though there may not be a canonical choice of the inclusion $\widetilde{K}\crt(\fS^d) \rightarrow K\crt(\fS^d)$.

We will complete the proof by showing $b = -d-2$.
As in the $d= 2$ case above, consider the short exact sequence 
$$0 \rightarrow S^d \C \xrightarrow{i} \fS^d \xrightarrow{\pi} \fS^{d-1} \rightarrow 0 \;$$
and the resulting long exact sequence
$$ \cdots \rightarrow K\crt( S^d \C) \xrightarrow{i_*} K\crt(\fS^d) \xrightarrow{\pi_*} K\crt(\fS^{d-1}) \xrightarrow{\partial} 
	K\crt( S^d \C) \xrightarrow{} \cdots  \; .$$
As we have seen, both $K\crt(\fS^d)$ and $K\crt(\fS^{d-1})$ have a summand isomorphic to $K\crt(\R)$ and $\pi_*$ restricts to an isomorphism on these summands. 
Eliminating these summands from the long exact sequence, we obtain
$$ \cdots \rightarrow K\crt( S^d \C) \xrightarrow{i_*} \widetilde{K}\crt(\fS^d) \xrightarrow{\pi_*} \widetilde{K}\crt(\fS^{d-1}) \xrightarrow{\partial} 
	K\crt( S^d \C) \xrightarrow{} \cdots  \; .$$
(Here we make an appropriate choice of the complimentary summands in such a way that $\pi_*$ carries $ \widetilde{K}\crt(\fS^d)$ to $ \widetilde{K}\crt(\fS^{d-1})$.)
	
We know that $\widetilde{K}\crt(\fS^2) \cong \Sigma^{-4} K\crt(\R) = K\crt(S^{-4} \R)$. Assume for induction that $\widetilde{K}\crt(\fS^{d-1}) \cong K\crt(S^{-d-1} \R)$ for $d \geq 3$. So
$$\widetilde{KO}_*(\fS^d) = (\Z, \Z_2, \Z_2, 0, \Z, 0, 0, 0)$$
in degrees $d+1, \dots, d+8 \pmod 8$.
Then the long exact sequence above has 
\begin{multline*} \cdots \rightarrow 
	{KO}_{d-1}(S^d \C) \rightarrow
	\widetilde{KO}_{d-1}(\fS^d) \rightarrow
	\widetilde{KO}_{d-1}(\fS^{d-1}) \rightarrow  \\
{KO}_{d-2}(S^d \C) \rightarrow
	\widetilde{KO}_{d-2}(\fS^d) \rightarrow
	\widetilde{KO}_{d-2}(\fS^{d-1}) \rightarrow \cdots 
\end{multline*}
which can be rewritten, using the induction hypothesis, as
\[  \cdots \rightarrow 0 \rightarrow
	\widetilde{KO}_{d-1}(\fS^d) \rightarrow
	0  \rightarrow  
\Z  \rightarrow
	\widetilde{KO}_{d-2}(\fS^d) \rightarrow
	0 \rightarrow \cdots 
\]

From this it follows immediately that $\widetilde{KO}_{d-2}(\fS^d)  = \Z$ and $\widetilde{KO}_{d-1}(\fS^d) = 0$.
Then the only possibility consistent with this information is
$$\widetilde{KO}_*(\fS^d) = (\Z, \Z_2, \Z_2, 0, \Z, 0, 0, 0)$$
in degrees $d+2, \dots, d+9 \pmod 8$.
Thus $\widetilde{K} \crt(\fS^d) = K\crt(S^{-d-2} \R)$.
\end{proof}

\begin{thm} \label{KK}
The real $K$-homology $KKO_*(\fS^1, \R)$ satisfies
\[KKO_*(\fS^d, \R) = 
	\begin{cases} \Sigma^{-1} KO_*( \fS^1) & d = 1 \\
				KO_*(\R) \oplus \Sigma^{d+2} KO_{*}(\R) & d \geq 2 . 
	\end{cases} \]
\end{thm}

\begin{proof}
In all cases, since $K\crt(\fS^d)$ is a free $\mathcal{CRT}$-module, the Universal Coefficient Theorem (Theorem~1.1 of \cite{Boer2004}) implies that
$KKO_*( \fS^1, A) \cong \hom\scrt( K\crt(\fS^1), K\crt(A))_i$ for any real $C \sp *$-algebra $A$ 
(this designates the group of $\mathcal{CRT}$-module morphisms of degree $i$). 

Furthermore, $K\crt(\fS^1) = K\crt(T)$ has a single free generator, which lives in the self-conjugate part of $K\crt(\fS^1)$ in degree -1. Therefore
\begin{align*}
KKO_i( \fS^1, \R) &= \hom_{\scriptscriptstyle {\it CRT}}( K\crt(\fS^1), K\crt(\R))_i \\
	&= KT_{i-1}(\R) \\
	&= KO_{i-1}(\fS^1) \; .
\end{align*}


For $d \geq 2$, from Theorem~\ref{fS^d} we know $K\crt(\fS^d)$ has two free generators which both live the real part, in degrees 0 and $d+2$. Therefore,
\begin{align*}
KKO_i( \fS^d, \R) &= \hom_{\scriptscriptstyle {\it CRT}}( K\crt(\fS^d), K\crt(\R))_i   \\
	&=  \hom_{\scriptscriptstyle {\it CRT}}( K\crt(\R) \oplus \Sigma^{-d-2} K\crt(\R), K\crt(\R))_i   \\
	&= KO_i(\R) \oplus KO_{i+d+2}(\R)  \; .
\end{align*}
\end{proof}


\section{Unitary Generators in Low Dimensions} \label{section-lowdim}

In this section, we present unitary elements that generate each of the groups of $K\crr(\fS^d)$, for $1 \leq d \leq 4$, following the unitary picture of real $K$-theory summarized in Table~\ref{unitaryTable}. 
We do this explicitly for the low dimensional cases because the resulting formulas comprise likely candidates for topological markers of physical systems, with an involution. In the following section, we will present an algorithm building these formulas in arbitrary dimension.
For $d = 2,3,4$ we write $K\crr(\fS^d) \cong K\crr(\R) \oplus \widetilde{K}\crr(\fS^d)$ and it suffices to focus on generators of the groups in the second summand, since the first summand is associated with the constant functions on the sphere and the relevant generators can be determined by the inclusion $\R \hookrightarrow \fS^d$. We will take for granted the isomorphism classes of all the groups $KO_i(\fS^d)$ and $KU_i(\fS^d)$ from Table~\ref{lowdim} in the previous section.

\vspace{2cm}

\begin{prop}
Table \ref{fS1unitary} shows unitary representatives of generators for $KO_i(\fS^1)$.
\end{prop}

\begin{table} 
\caption{Unitaries for $K\crr(\fS^1)$} \label{fS1unitary}
$$\begin{array}{|c|c| c|c|}  \hline
  & \text{isomorphism class} & 
 	\multicolumn{2}{|c|}{ \text{unitary representing a generator}}  \\ \hline \hline
 KU_0( \fS^1) & \Z & y_0 = 1_2 &  1_2 \\ \hline
 KU_1(\fS^1) & \Z & y_1 = \exp(2 \pi i t)  &  z \\ \hline \hline
KO_0(\fS^1) & \Z & x_{0} = 1_2 & 1_2  \\ \hline 
KO_1(\fS^1) & \Z_2 & x_{1} = -1 & -1  \\ \hline
KO_2(\fS^1) & 0 &&  \\ \hline 
KO_3(\fS^1) & \Z & x_3 = \begin{cases} \diag(\exp({4 \pi i t}), 1) &t \in [0, 1/2] \\
				\diag(1, \exp({4 \pi i t}))  & t \in [1/2, 0] \end{cases} 
&  \begin{cases} \diag(z^2, 1) &z \in (S^1)^+\\
				\diag(1, z^2) &z \in (S^1)^- \end{cases}   \\ \hline
KO_4(\fS^1) & \Z & 
	x_4 = \diag \left( \left(  \begin{smallmatrix} \cos 2 \pi t & \sin 2 \pi t   \\ \sin 2 \pi t & -\cos 2 \pi t &   \end{smallmatrix} \right) , 1_2 \right) & 
	 \diag \left( \left(  \begin{smallmatrix} x & y   \\ y & -x &   \end{smallmatrix} \right) , 1_2 \right) \\ \hline  
KO_5(\fS^1) & \Z_2 & 
	x_5 = \diag \left( \left(  \begin{smallmatrix} \cos 2 \pi t & \sin 2 \pi t   \\ \sin 2 \pi t & -\cos 2 \pi t &   \end{smallmatrix} \right) , 1_2 \right) & 
	 \diag \left( \left(  \begin{smallmatrix} x & y   \\ y & -x &   \end{smallmatrix} \right) , 1_2 \right) \\ \hline  
KO_6(\fS^1) & 0 &&  \\ \hline 
KO_{7}(\fS^1) & \Z & x_{-1} = \exp(4 \pi i t) &  z^2   \\ \hline 
\end{array}$$

\end{table}

In this table, the third column represents $\fS^1$ using functions on $t \in [0,1]$ that satisfy $f(0) = f(1)$, with the involution 
$$f^\tau(t) = \begin{cases} f(t+0.5)  & t \in \left[ 0, \tfrac 12  \right]   \\ f(t-0.5)  & t \in \left[ \tfrac 12 , 1 \right] \; . \end{cases} $$
The third column represents $\fS^1$ equivalently using functions on $z = x + iy \in S^1 \subset \C$ with the involution 
$f^\tau(z)  = {f(-z)}$
In the line for $KO_3$, note that $(S^1)^+$ and $(S^1)^-$ are the top and bottom half of the circle $S^1 \subset \C$ respectively.

\begin{proof}
The complex groups $KU_i(\fS^1) = K_i(C(S^1)) = (\Z, \Z)$ are well known and the unitary generators $y_0$ and $y_1$ can be found in Section 4 of  \cite{Boer2020}. The abstract isomorphism classes of the groups $KO_i(\fS^1) = KO_i(T)$ are known from Table~3 of \cite{Boer2002}.
Also, from Table~3 of \cite{Boer2002} we know that $c_i$ is an isomorphism for $i = 0,4$ and is multiplication by 2 for $i = -1, 3$. These facts will be used below. 

The unital inclusion $\R \rightarrow \fS^1$ induces an isomorphism on $KO_i$ for $i = 0,1$ and using this map the classes of $x_0$, $x_1$ can be identified as shown in the table, using the known representatives of $KO_i(\R)$. 

Now the formula for $c$ from \cite{Boer2002} is $c_i[u] = [u]$ for all $i$. 
For $x_4$, first vertify that the given $x_4$ is a self-adjoint unitary satisfying $x^{\sharp \otimes \tau} = x$, so we have $[x_4] \in KO_4(\fS^1)$. 
Since $c_4 \colon KO_4(\fS^1) \rightarrow KU_4(\fS^1)$ is an isomorphism and we have $c_4[x_4] = [x_4]$, it suffices to show that $x_4$ represents a generator of $KU_4(\fS^1) = K_0(C(S^1, \C)) = \Z$. But now there is also an isomorphism $\pi_* \colon K_0(C(S^1, \C)) \rightarrow K_0(\C)$ given by point evaluation, so it suffices to observe that $[\pi(x_4)] = [\diag(1, -1, 1, 1)] = [\diag(1,1)]$ generates $K_0(\C) = \Z$.

The formula $\eta_4[u] = [u]$ from \cite{Boer2020} immediately implies that the given $x_5$ is a generator of $KO_5(\fS^1) = \Z_2$.

For $x_{-1}$ and $x_3$, use the fact that for $i = -1, 3$ the map  $c_{i} \colon KO_{i}(\fS^1) \rightarrow KU_{i}(\fS^1)$ is multiplication by 2 from $\Z \rightarrow \Z$. 
We verify that in both cases, the given element is a unitary $x_i$ that satisfies the symmetry to represent an element in $KO_i(\fS^1)$. In particular check that $x_{-1}^\tau = x_{-1}$ and that $(x_3)^{\sharp \otimes \tau} = x_3$. So we have
$c_{-1}([x_{-1}]) = [x_{-1}] \in KU_1(\fS^1)$; and as classes in $KU_1(\fS^1)$ we have $[x_{-1}] = [\diag(\exp(2 \pi it), \exp(2 \pi it) ) ] = 2[ y_1]$ through an obvious path of unitaries in $M_2(\fS^1 \otimes \C) = M_2(C(S^1))$. Therefore $[x_{-1}]$ must be a generator of $KO_{-1}(\fS^1) = \Z$. 
A very similar argument also shows $c_3[x_3] = 2[y_1]$ so that $[x_3]$ must be a generator of $KO_3(\fS^1) = \Z$.
\end{proof}

\vspace{.5cm}

\begin{prop}
Table \ref{fS2unitary} shows unitary representatives of generators for $\widetilde{K}\crr_i(\fS^2)$.
\end{prop}

\begin{table} 
\caption{Unitaries for $K\crr(\fS^2)$} \label{fS2unitary}
$$\begin{array}{|c|c|c|}  \hline
  & \text{isomorphism class} & { \text{unitary representing a generator}}  \\ \hline \hline
\TT \BB \widetilde{KU}_0( \fS^2)  & \Z &  y_0 = \sm{x}{y + iz}{y - iz}{-x}  \\ \hline
\TT \BB \widetilde{KU}_1(\fS^2) & 0 &     \\ \hline \hline
\TT \BB \widetilde{KO}_{0}(\fS^2)  & \Z & 
	x_0 =  \left(  \begin{smallmatrix}  0 & iz & ix & iy \\ -iz &  0 & iy & -ix \\ -ix & -iy & 0 & iz \\ -iy & ix & -iz & 0 
						\end{smallmatrix} \right) 
					\\ \hline 
\TT \BB \widetilde{KO}_1(\fS^2)  & 0 &   \\ \hline 
\TT \BB \widetilde{KO}_2(\fS^2)  & 0 &    \\ \hline
\TT \BB\widetilde{KO}_3(\fS^2)  & 0 &  \\ \hline 
\TT \BB \widetilde{KO}_4(\fS^2)  & \Z & x_4 = 
		\left( \begin{smallmatrix} x & y + iz   \\ y - iz & -x  \end{smallmatrix} \right) \\  \hline
\TT \BB \widetilde{KO}_5(\fS^2) & \Z_2 & x_5 =
	\left( \begin{smallmatrix} x & y + iz   \\ y - iz & -x  \end{smallmatrix} \right) \\  \hline
\TT \BB \widetilde{KO}_6(\fS^2) & \Z_2 & 
	x_6 =  \left( \begin{smallmatrix} 0 & 0 & ix & -z + iy \\ 0 & 0 & z + iy & -ix
			\\ -ix & z - iy & 0 & 0 \\  -z - iy & ix & 0 & 0
				\end{smallmatrix} \right)	 \\ \hline 
\TT \BB \widetilde{KO}_7(\fS^2)   & 0 &  \\ \hline 
\end{array}$$
\end{table}

In this table, the elements of $\fS^2$ are written as functions on $(x,y,z) \in S^2 \subset \R^3$ with the involution $f^\tau(x,y,z)  = {f(-x,-y,-z)}$.
Note that as in the previous proof, we will take for granted the isomorphism class of the groups shown (from Table~3) as well as the action of the natural transformations such as $r_i$, $c_i$, and $\eta_i$ on those groups. The behavior of these natural transformations can be deduced just from the long exact sequence relating real and complex $K$-theory. It can also be found in Table~1 of \cite{Boer2002}. We will also make use of the various formulas for these natural transformations in terms of unitaries, which are established in \cite{Boer2020}.

\begin{proof}
The element $y_0 \in \widetilde{KU}_0(\fS^2)$ is the well-known Bott unitary for the sphere, see Example~7.1.3 of \cite{BL} or Proposition~1 of \cite{SBS}.

We can check directly that $x_4^{\sharp \otimes \tau} = x_4 = x_4^*$, so $[x_4]$ is a legitimate element of 
$\widetilde{KO}_4(\fS^2)$. Furthermore, in this context we know that $c_4$ is an isomorphism, so the fact that $c_4[x_4] = [y_4]$ verifies that $[x_4]$ is a generator of $\widetilde{KO}_4(\fS^2) = \Z$. We also know that $\eta_4 \colon \Z \rightarrow \Z_2$ is surjective and we have the formula $\eta_4[u] = [u]$ from Theorem~5.1 of \cite{Boer2020}. Therefore $x_5$ must be as given.

We know that $r_0 \colon \widetilde{KU}_0( \fS^2) \rightarrow \widetilde{KO}_0( \fS^2)$ is an isomorphism, and from Theorem~3.2 of \cite{Boer2002}, we know that a formula for $r_0$ is given by
$$r_0[u] = \left[ \left( \begin{smallmatrix} a & ib \\ -ib & a  \end{smallmatrix} \right) \right] \quad 
	\text{where $a = \tfrac{1}{2} (u + u^\tau)$ and $b = \tfrac{1}{2} (u - u^\tau)$} \; .$$
Using this we calculate $r_0[y_0]$ and obtain the given formula for $x_0$.

Finally, to get the formula for $x_6$ we use the fact that $r_6$ is surjective in this case and the formula for $r_6[y_0]$ from \cite{Boer2020}. That formula is
$$r_6[u] = \left[ \left( \begin{smallmatrix} a & ib \\ -ib & a  \end{smallmatrix} \right) \right] \quad 
	\text{where $a = \tfrac{1}{2} (u - u^{\sharp \otimes \tau})$ and $b = \tfrac{1}{2} (u + u^{\sharp \otimes \tau})$}  \;.$$
Noting that $y_0$ satisfies $y_0^{\sharp \otimes \tau} = y_4$, we obtain $a = 0$ and $b = x_4$.
\end{proof}

\begin{prop}
Table \ref{fS3unitary} shows unitary representatives of generators for $\widetilde{K}\crr_i(\fS^3)$.
\end{prop}

\begin{table} 
\caption{Unitaries for $K\crr(\fS^3)$} \label{fS3unitary}
$$\begin{array}{|c|c|c|}  \hline
  & \text{isomorphism class} & { \text{unitary representing a generator}}  \\ \hline \hline
 \TT \BB \widetilde{KU}_0( \fS^3) & 0 &  \\ \hline
 \TT \BB \widetilde{KU}_1(\fS^3) & \Z & y_1 = \sm{x + iw}{y + iz}{y - iz}{-x + iw}     \\ \hline \hline
 \TT \BB \widetilde{KO}_{0}(\fS^3) & 0 & 			\\ \hline 
 \TT \BB\widetilde{KO}_1(\fS^3) &  \Z &   
	x_{1} =  i \left( \begin{smallmatrix} w & z & x & y \\ -z & w & y & -x \\ -x & -y & w & z \\ -y & x & -z & w \end{smallmatrix} \right)
	\\ \hline 
 \TT \BB\widetilde{KO}_2(\fS^3) & 0 &       \\ \hline
 \TT \BB\widetilde{KO}_3(\fS^3) & 0 &  \\ \hline 
 \TT \BB\widetilde{KO}_4(\fS^3) & 0 &  \\ \hline
 \TT \BB\widetilde{KO}_5(\fS^3) & \Z & 
		x_5= \sm{x + iw}{y + iz}{y - iz}{-x + iw}   \\ \hline 
 \TT \BB\widetilde{KO}_6(\fS^3) & \Z_2 & 
	x_6= 	  \left( \begin{smallmatrix} 0 & 0 & -w + ix & -z + iy \\ 0 & 0 & z + iy & -w - ix \\ w + ix & -z + iy & 0 & 0 \\ z + iy & w - ix & 0 & 0 \end{smallmatrix} \right)
			\\ \hline 
 \TT \BB\widetilde{KO}_7(\fS^3) & \Z_2 &  
	x_7 = \left( \begin{smallmatrix} 0 &  iz & -w+ix & iy \\ -iz & 0 & iy & -w-ix \\ w-ix & -iy & 0 & iz \\ -iy &w + ix & -iz & 0 \end{smallmatrix} \right)
			\\ \hline 
\end{array}$$
\end{table}

\begin{proof}
Proposition~1 of \cite{SBS} shows that $y_1$ is a unitary representing a non-trivial complex $K$-theory class generating $K_1(C(S^3)) = \Z$. 

For $i = 5$, first check that the given $x_5$ satisfies the correct relation $(x_5)^{\sharp \otimes \tau} = x_5^*$ and thus it represents an element of ${KO}_5(\fS^3)$. We furthermore know that $c_5$ must carry a generator of $\widetilde{KO}_5(\fS^3)$ to a generator of $\widetilde{KU}_5(\fS^3) = \Z$ (this can be deduced from the long exact sequence relating real and complex $K$-theory).
We see directly that $c[x_5] = [y_1]$ so it follows that $[x_5]$ is a generator of $\widetilde{KO}_5(\fS^3) = \Z$.

For $i = 1,7$ it is known that $r_i$ carries a generator of $\widetilde{KU}_i(\fS^3)$ to a generator of $\widetilde{KO}_i(\fS^3)$, that is $r_i[y_i] = [x_i]$. So we can identify representatives $x_1$ and $x_7$ of $\widetilde{KO}_i(\fS^3)$ using the formulas for $r_1$ and $r_7$ from Theorem~3.2 of \cite{Boer2020}.

For $x_6$, we know that $\eta_5$ carries a generator of $\widetilde{KO}_5(\fS^3)$ to a generator of $\widetilde{KO}_6(\fS^3)$. Then we use the formula for $\eta_5$ from \cite{Boer2020} and the fact that $\eta_5[x_5] = [x_6]$ to obtain the given formula for $x_6$.
\end{proof}

\begin{prop}
Table \ref{fS4unitary} shows unitary representatives of generators for $\widetilde{K}\crr_i(\fS^4)$.
\end{prop}

\begin{table} 
\caption{Unitaries for $K\crr(\fS^4)$} \label{fS4unitary}

$$\begin{array}{|c|c|c|}  \hline
  & \text{isomorphism class} & { \text{unitary representing a generator}}  \\ \hline \hline
 \TT \BB \widetilde{KU}_0( \fS^4) & \Z & 
 	y_0 = \left( \begin{smallmatrix} 
			v & 0 & x + iy & {w + iz} \\
			0 & v & {-w + iz} & {x - iy}  \\
			x - iy & -w - iz & -v & 0 \\
			{w - iz}& {x + iy }  & 0 & -v
			  \end{smallmatrix} \right)  \\ \hline			  
\TT \BB  \widetilde{KU}(\fS^4) & 0 &     \\ \hline \hline
\TT \BB \widetilde{KO}_{0}(\fS^4) & \Z_2 &  x_0 = 
	\left( \begin{smallmatrix}  a & ib \\ -ib & a \end{smallmatrix} \right)  \\ \hline	
\TT \BB \widetilde{KO}_1(\fS^4) & 0 &  \\ \hline 
\TT \BB \widetilde{KO}_2(\fS^4) &  \Z &   x_2 = 			
	\left( \begin{smallmatrix} b & ia \\ -ia & b \end{smallmatrix} \right)  \\ \hline
\TT \BB \widetilde{KO}_3(\fS^4) & 0 &  \\ \hline 
\TT \BB \widetilde{KO}_4(\fS^4) &  &  \\ \hline
\TT \BB \widetilde{KO}_5(\fS^4) & 0 &  \\ \hline 
\TT \BB \widetilde{KO}_6(\fS^4) & \Z & x_6 =  \left( \begin{smallmatrix} 
			v & 0 & x + iy & {w + iz} \\
			0 & v & {-w + iz} & {x - iy}  \\
			x - iy & -w - iz & -v & 0 \\
			{w - iz}& {x + iy }  & 0 & -v
			  \end{smallmatrix} \right)  \\ \hline	
\TT \BB \widetilde{KO}_7(\fS^4) & \Z_2 & x_7 =  \left( \begin{smallmatrix} 
			0 & -iv & -w + iz & -y-ix \\
			iv & 0 & -y+ix & w + iz \\
			w-iz & y - ix & 0 & iv \\
			y+ ix & -w - iz & -iv & 0  \\
			  \end{smallmatrix} \right)  \\ \hline	
\end{array}$$
\[ \text{where } a = i
	\left(  \begin{smallmatrix} 0 & 0 & y & z \\ 0 & 0 & z & -y \\ -y & -z & 0 & 0 \\ -z & y & 0 & 0 \end{smallmatrix} \right) 
		 \text{and }  
	b = \left( \begin{smallmatrix}  v & 0 & x & w \\ 0 & v & -w & x \\ x & -w & -v & 0 \\ w & x & 0 & -v \end{smallmatrix} \right) \]
\end{table}

\begin{proof}
Let 
$$\sigma_1 = \begin{pmatrix} i & 0 \\ 0 & -i \end{pmatrix}
\quad \sigma_2 = \begin{pmatrix} 0 & 1 \\ -1 & 0 \end{pmatrix}
\quad \sigma_3 = \begin{pmatrix} 0 & i \\ i & 0 \end{pmatrix}
$$
and
$$
\Up_i = \begin{cases} \begin{pmatrix} 0 & \sigma_i \\ \sigma_i^* & 0 \end{pmatrix} & \text{for $i = 1,2,3$} \\ ~ \\
 \begin{pmatrix} 0 & 1_2 \\ 1_2 & 0 \end{pmatrix} & i = 4 \\ ~ \\
 \begin{pmatrix}  1_2 & 0  \\ 0 & -1_2  \end{pmatrix} & i = 5 \; .
\end{cases}$$
Then we check that $\Up_i^* = \Up_i$ and $\Up_i \Up_j + \Up_j \Up_i = 2 \delta_{i,j}$ for all $i,j$. Therefore, by Proposition~1 of \cite{SBS}, it follows that
$$y_0 = x \Up_1 + y \Up_2 + z \Up_3 + w \Up_4 + v \Up_5$$
represents a generator of $\widetilde{KU}_0(S^4)$. Note that this is a different set of Clifford elements $\Up_i \in M_4(\C)$ to that described in \cite{SBS}. The representatives here are chosen to satisfy $\Up_i^{\sharp \otimes {\rm{Tr}}} = \Up_i$ for all $i$. 
 Any representation of $M_4(\C)$ is unitarily equivalent to any other, so the resulting class $[y_0]$ will be the same up to sign.
 (We will see in Section~\ref{Section:S^n unitaries} how to construct the appropriate set of Clifford generators in general. The matrices $\Up_i$ that appear here can be obtained from that construction.)

Since $\Up_i^{\sharp \otimes {\rm{Tr}}} = \Up_i$ for all $i$, it follows that
$x_6$ satisfies $x_6^{\sharp \otimes \tau} = - x_6$, so $[x_6]$ is a legitimate element of $\widetilde{KO}_6(\fS^4)$. Then the fact that $c_6[x_6] = [y_0]$ confirms that $[x_6]$ is a generator of $\widetilde{KO}_6(\fS^4) = \Z$.

To find $x_7$, use the formula for $\eta_6$ from \cite{Boer2020}. (Note that there is a mistake in the formula for $\eta_2$ and $\eta_6$ in \cite{Boer2020}. Those formulas should involve the matrices $X_{1,n}$ instead of $X_{2,n}$.)

For $x_0$ and $x_2$, we know that $r_i$ is surjective for $i = 0,2$. Hence, applying the appropriate formulas for $r_0$ and $r_2$ (from \cite{Boer2020}) to $y_0$ 
results in the given formulas for $x_0$ and $x_2$. (We note that the formulas for $r_0$ and $r_2$ require an extra step of conjugation by a matrix $X_{1, n}$ at the end. However, this is not necessary since the matrix $X_{1,8}$ is in $SO(8)$ which is connected. Conjugation by $X_{1,8}$ does not change the $K$-theory class.)
\end{proof}

\section{$K$-Theory Generators in Higher Dimensions} \label{Section:S^n unitaries}

It is clear that there is a limit to how much progress can be made with this approach of finding formulas for unitary generators for $K$-theory for each $\fS^d$ separately as in the previous section. In this section, we describe a more general approach to finding a representative of the key generator of $\widetilde{KO}_{d+2}(\fS^d)$.

For each $d$ there is an element $g_d \in \widetilde{KO}_{d+2}(\fS^d) = \Z$ that generates $\widetilde{KO}_*(\fS^d)$ as a free $KO_*(\R)$-module, and that generates $\widetilde{K}\crr(\fS^d)$ as a $\mathcal{CR}$-module. In this section, for each $d$ we will describe how to construct a unitary defined on the sphere, with the proper symmetry, that represents $g_d$. We will not explicitly describe how to obtain unitary generators for the other non-trivial groups of $\widetilde{KO}_*(\fS^d)$. We suffice it to say that these generators can all be obtained from the key generator in $\widetilde{KO}_{d+2}(\fS^d)$ using the natural transformations. Specifically, if $g_d$ is a generator of $\widetilde{KO}_{d+2}(\fS^d)$, 
then $\eta(g_d)$ is a generator of $\widetilde{KO}_{d+3}(\fS^d) = \Z_2$,
 $\eta^2(g_d)$ is a generator of $\widetilde{KO}_{d+4}(\fS^d) = \Z_2$,
 and $\xi(g_d)$ is a generator of $\widetilde{KO}_{d+6}(\fS^d) = \Z$. 
 Hence, when we have a unitary representing $g_d$, we can obtain a formula for the unitaries representing $\eta(g_d)$, $\eta^2(g_d)$, and $\xi(g_d)$ using the formulas for $\eta$ and $\xi$ found in \cite{Boer2020}.

Our main results are Construction~\ref{construction} and Theorem~\ref{main-thm} below. Construction~\ref{construction} is presented as an algorithm to produce unitaries 
and Theorem~\ref{main-thm} states that the unitaries from this construction actually represent generators of the correct $K$-theory groups. This algorithm builds on the construction for complex $K$-theory (from \cite{SBS}), which is re-stated as Part (0) of Construction~\ref{construction}. The real $K$-theory elements are given in Parts (1)-(4). 
First we present some preliminaries on Clifford generators. 

\begin{defn}
Let $k$ be a positive odd integer. A list of self-adjoint matrices $a_1, \dots, a_k \in M_n(\C)$ for some $n$ is called a list of ``complex $k$-Clifford generators" if they satisfy 
$$a_i a_j + a_j a_i = 2 \delta_{i,j} \; $$
for all $i,j$. These elements form an irreducible representation of the the Clifford algebra $\C_k$.
\end{defn}

\begin{const}  \label{construction1}
Let $k$ be an positive odd integer and let $n = 2^{(k-1)/2}$. We consider the following to be the ``standard" list of complex self-adjoint $k$-Clifford generators in $M_n(\C)$, following the construction in Section 2 of \cite{SBS}, constructed inductively.
\begin{itemize} 
\item For $k = 1$, we have $\Gamma_{1,1} = 1$.
\item For $k = 3$, we have 
$$\Gamma_{1,3} = \begin{pmatrix} 0 & 1 \\ 1 & 0 \end{pmatrix} \; , \quad
\Gamma_{2,3} =  \begin{pmatrix} 0 & i \\ -i & 0 \end{pmatrix} \; , \quad
\Gamma_{3,3} =  \begin{pmatrix} 1 & 0 \\ 0 & -1 \end{pmatrix} \; .$$
\item For $k \geq 3$, let $\Gamma_{1,k}, \dots, \Gamma_{k,k}$ be the standard list of complex self-adjoint $k$-Clifford generators in $M_n(\C)$, and then define $\Gamma_{1,k+2}, \dots, \Gamma_{k+2,k+2} \in M_{2n}(\C)$ by
\begin{align*}
\Gamma_{i, k+2} &= \begin{pmatrix} 0 & \Gamma_{i, k} \\ \Gamma_{i,k} & 0 \end{pmatrix}  \; , && i = 1, \dots, k, \\
\Gamma_{k+1,k+2} &=  \begin{pmatrix} 0 & i I_{n} \\ -i I_{n} & 0 \end{pmatrix} \; , \\
\Gamma_{k+2, k+2} &=  \begin{pmatrix} I_{n} & 0 \\ 0 & -I_{n}  \end{pmatrix}  \; .\\
\end{align*}
\end{itemize}
\end{const}

\begin{lemma}\label{Lemma:standardlist}
For positive odd integers $k$, the standard list of complex self-adjoint $k$-Clifford generators satisfies
$$
(\Gamma_{i,k})\T = (-1)^{i+1} \Gamma_{i,k} 
\quad \text{and} \quad
(\Gamma_{i,k})\sT = 
\begin{cases} 
-\Gamma_{i,k} & i \leq 3 \\
(-1)^{i+1} \Gamma_{i,k} & \rm{otherwise.}
\end{cases} $$
\end{lemma}

\begin{proof}
Check this directly for $k = 3$ and then proceed by induction.
\end{proof}

\begin{lemma} \label{transpose}
Let $a_1, \dots, a_k, b_1, \dots, b_\ell$ be a list of complex self-adjoint $(k + \ell)$-Clifford generators satisfying
$$(a_i)\T = a_i \quad \text{and} \quad (b_i)\T = -b_i \; .$$
Define 
$$\widetilde a_i = i (a_1 a_2 \dots \widehat a _i \dots a_k) \text{~for $1 \leq i \leq k$} 
\quad \text{and} \quad \widetilde b _i = i (b_1 b_2 \dots \widehat{b_i} \dots b_\ell) \text{~for $1 \leq i \leq \ell$} \; .
$$
Then
\begin{enumerate}
\item If $\ell \equiv 0 \pmod 4$, then
$a_1, \dots, a_k, \widetilde b_1, \dots, \widetilde b_\ell$ is a list of complex self-adjoint $(k + \ell)$-Clifford generators 
and $(\widetilde{b}_i)\T = \widetilde{b}_i$
\item If $k \equiv 0 \pmod 4$, then
$\widetilde{a}_1, \dots, \widetilde{a}_k, b_1, \dots, b_\ell$ is a list of complex self-adjoint $(k + \ell)$-Clifford generators 
and $(\widetilde{a}_i)\T = -\widetilde{a}_i$
\end{enumerate}
\end{lemma}

\begin{lemma} \label{sharp}
The same statements in Lemma~\ref{transpose} hold if we replace $x\T$ with $x\sT$ throughout.
\end{lemma}

For any positive integer $k$, let $\Delta(k) =\tfrac{k(k+1)}{2}$ be the $k$th triangular number. Note that $\Delta(k-1)$ is the number of pair-wise adjacent transpositions involved to transform a product $x_1 \dots x_k $ into $ x_k \dots x_1$ and that $\Delta(k)$ is even if and only if $k \equiv 3$ or $k \equiv 4 \pmod 4$.

\begin{proof}[Proof of Lemmas~\ref{transpose} and \ref{sharp}.]
For part (1), check first that
\begin{align*}
(\widetilde b_i)^* &= -i (b_1 b_2 \dots \widehat{b_i} \dots b_\ell)^* \\
	&= -i (b_\ell \dots \widehat{b_i} \dots b_2 b_1) \\
	&= -i (-1)^{\Delta(\ell-2)} (b_1 b_2 \dots \widehat{b_i} \dots b_\ell) \\
	&= \widetilde b_i 
\end{align*}
since $\ell-2 \equiv 2 \pmod 4$. From this it follows that $\widetilde b_i$ are self-adjoint unitaries. In particular $(\widetilde b_i)^2 = 1$.

Similarly we have
\begin{align*}
(\widetilde b_i)\T &=  i (b_1 b_2 \dots \widehat{b_i} \dots b_\ell)\T \\
	& = i (-1)^{\Delta(\ell-2)} (b_1\T b_2\T \dots \widehat{b_i} \dots b_\ell\T) \\
	&= i (-1)^{\Delta(\ell-2)} (-1)^{\ell-1} (b_1 b_2 \dots \widehat{b_i} \dots b_\ell) \\
	&= \widetilde b_i \; .
\end{align*}
We also compute, for all $1 \leq i \leq k, 1 \leq j \leq \ell$
\begin{align*}
a_i \widetilde b_j &= i a_i (b_1 \dots \widehat{b}_j \dots b_\ell) \\
	&= i (-1)^{\ell-1} (b_1 \dots \widehat{b}_j \dots b_\ell) a_i \\
	&= - \widetilde b_j a_i
\end{align*}
and, for all $1 \leq i < j \leq \ell$
\begin{align*}
\widetilde b_i \widetilde b_j 
	&= i^2 (b_1 \dots \widehat{b}_i \dots b_\ell)(b_1 \dots \widehat{b}_j \dots b_\ell) \\
	&= i^2 (-1)^{\ell-1} (-1)^{\ell-2} (b_1 \dots \widehat{b}_j \dots b_\ell)(b_1 \dots \widehat{b}_i \dots b_\ell) \\
	&= - \widetilde b_j \widetilde b_i \; .
\end{align*}
This proves Part (1).

For Part (2) we set $\widetilde a_i = i (a_1 a_2 \dots \widehat a _i \dots a_k)$. 
Then we have 
\begin{align*}
(\widetilde a_i)\T &=  i (a_1 a_2 \dots \widehat{a_i} \dots a_k)\T \\
	& = i (-1)^{\Delta(k-2)} (a_1\T a_2\T \dots \widehat{a_i} \dots a_k\T) \\
	&= i (-1)^{\Delta(k-2)}  (a_1 a_2 \dots \widehat{a_i} \dots a_k) \\
	&= - \widetilde a_i; .
\end{align*}
The proof is otherwise the same  as for Part (1). 

The proof of Lemma~\ref{sharp} is identical to that of Lemma~\ref{transpose}, since both involutions ${\rm Tr}$ and $\sharp \otimes {\rm Tr}$ have the same formal properties -- they are linear and antimultiplicative.
\end{proof}

\begin{const} ~ \label{construction}

For $k$ odd, this construction shows how to produce particular unitaries $Q_{k-1}\pr$ and $U_k\pr$, which ultimately we show to represent the generators of the $KO$-groups indicated. We include here the key statements about which $KO$-groups they live in for later reference;  these statements will be justified by Theorem~\ref{main-thm} below.

We also include the construction in the complex case for completeness. In each case, start with the standard complex Clifford generators $\Gamma_{1,k}, \dots, \Gamma_{k,k} \in M_n(\C)$ where $n = 2^{(k-1)/2}$ from Construction~\ref{construction1}.

\begin{enumerate}
\item[(0)] Let $k \equiv 1 \pmod 2$.
\begin{itemize} 
\item Let $Q_{k-1}(x) = \sum_{i = 1}^k x_i \Gamma_{i,k}$ and $U_k(x) = \sum_{i = 1}^k x_i \Gamma_{i,k} + x_{k+1} i I_n$.
\item Then $[Q_{k-1}] \in K_{0}(C(S^{k-1}), \C)$ and $[U_k] \in K_{1}(C(S^k), \C)$.
\end{itemize}

\vspace{.5cm}

\item[(1)] Let $k \equiv -1 \pmod 8$.
\begin{itemize}
\item Let $$  S=  \{j \in  \{1, \dots, k\} \mid \text{$j$ is odd} \} \; .$$
\item Define $\Up_{1,k}, \dots, \Up_{k,k} \in M_n(\C)$ by
$$\Up_{i,k} =
	\begin{cases} 
			 i \prod_{ j \in S \backslash \{i\}} \Gamma_{j,k} & i \in S \\ \hspace{.5cm} \Gamma_{i,k} & i \notin S 
	\end{cases} $$
\item Then $\Up\T_{i,k} =  - \Up_{i,k} \; $ for all $i$.
\item Let $Q\pr_{k-1}(x) = \sum_{i = 1}^k x_i \Up_{i,k}$ and $U\pr_k(x) = \sum_{i = 1}^k x_i \Up_{i,k} + x_{k+1} i I_n$.
\item Then $[Q\pr_{k-1}] \in KO_{0}(\fS^{k-1})$ and $[U\pr_k] \in KO_{1}(\fS^{k})$.
\end{itemize}

\vspace{.5cm}

\item[(2)] Let $k \equiv 1 \pmod 8$.
\begin{itemize}
\item Let $$  S=  \{j \in  \{1, \dots, k\} \mid \text{$j$ is even} \} \; .$$
\item Define $\Up_{1,k}, \dots, \Up_{k,k} \in M_n(\C)$ by
$$\Up_{i,k} =
	\begin{cases} 
			 i  \prod_{ j \in S \backslash \{i\}} \Gamma_{j,k} & i \in S \\ \hspace{.5cm} \Gamma_{i,k} & i \notin S 
	\end{cases} $$
\item Then $\Up\T_{i,k} =   \Up_{i,k} \; $ for all $i$.
\item Let $Q\pr_{k-1}(x) = \sum_{i = 1}^k x_i \Up_{i,k}$ and $U\pr_k(x) = \sum_{i = 1}^k x_i \Up_{i,k} + x_{k+1} i I_n$.
\item Then $[Q\pr_{k-1}] \in KO_{2}(\fS^{k-1})$ and $[U\pr_k] \in KO_{3}(\fS^{k})$.
\end{itemize}

\vspace{.5cm}

\item[(3)] Let $k \equiv 3 \pmod 8$.
\begin{itemize}
\item Let $$  S=  \{j \in  \{1, \dots, k\} \mid \text{$j$ is odd and $j \geq 5$} \} \; .$$
\item Define $\Up_{1,k}, \dots, \Up_{k,k} \in M_n(\C)$ by
$$\Up_{i,k} =
	\begin{cases} 
			 i  \prod_{ j \in S \backslash \{i\}} \Gamma_{j,k} & i \in S \\ \hspace{.5cm} \Gamma_{i,k} & i \notin S 
	\end{cases} $$
\item Then $\Up\sT_{i,k} =  - \Up_{i,k} \; $ for all $i$.
\item Let $Q\pr_{k-1}(x) = \sum_{i = 1}^k x_i \Up_{i,k}$ and $U\pr_k(x) = \sum_{i = 1}^k x_i \Up_{i,k} + x_{k+1} i I_n$.
\item Then $[Q\pr_{k-1}] \in KO_{4}(\fS^{k-1})$ and $[U\pr_k] \in KO_{5}(\fS^{k})$.
\end{itemize}

\vspace{.5cm}

\item[(4)] Let $k \equiv 5 \pmod 8$.
\begin{itemize}
\item Let $$  S=  \{j \in  \{1, \dots, k\} \mid \text{$j$ is even or $j \leq 3$} \}  \; . $$
\item Define $\Up_{1,k}, \dots, \Up_{k,k} \in M_n(\C)$ by
$$\Up_{i,k} =
	\begin{cases} 
			 i \prod_{ j \in S \backslash \{i\}} \Gamma_{j,k} & i \in S \\ \hspace{.5cm} \Gamma_{i,k} & i \notin S 
	\end{cases} $$
\item Then $\Up\sT_{i,k} =   \Up_{i,k} \; $ for all $i$.
\item Let $Q\pr_{k-1}(x) = \sum_{i = 1}^k x_i \Up_{i,k}$ and $U\pr_d(x) = \sum_{i = 1}^k x_i \Up_{i,k} + x_{k+1} i I_n$.
\item Then $[Q\pr_{k-1}] \in KO_{6}(\fS^{k-1})$ and $[U\pr_k] \in KO_{7}(\fS^{k})$.
\end{itemize}
\vspace{.5cm}
\end{enumerate}

Note that in each case $|S| \equiv 0 \pmod 4$, so that Lemma~\ref{transpose} and \ref{sharp} will apply (see Lemma~\ref{lemma:sigma-generators} below).
We also note that in the products that define $\Up_{i,k}$ for $i \in S$ in Parts (1)-(4) above, the order of the factors 
matters since the Clifford generators do not commute. The proper order is always taken to correspond to the natural order of increasing integer indices $j \in S \backslash \{ i \}$ from left to right. (Any different order would also do, as long as it was the same order used for all the $\Up_{i,k}$.)
\end{const}

In the complex case, we know that the elements $[Q_{k-1}]$ and $[U_k]$ from Part (0) of Construction~\ref{construction} represent generators of $K_{0}(C(S^{k-1}), \C) = \Z$ and $K_{1}(C(S^k), \C) = \Z$ respectively from Proposition~1 in \cite{SBS}. Lemmas~\ref{lemma:sigma-generators} and \ref{lemma:KO} below confirm the statements made in the second and forth bullet point of each item of Construction~\ref{construction}.

\begin{lemma} \label{lemma:sigma-generators}
Suppose that $k$ is odd. In each of the four cases of Construction~\ref{construction}, the constructed list $\Up_{1,k}, \dots, \Up_{k,k}$ consists of Clifford generators that satisfy the following:
\begin{enumerate}
\item for $k \equiv -1 \pmod 8$ and all $1 \leq i \leq k$, we have $\Up_{i,k}\T = -\Up_{i,k}$.
\item for $k \equiv 1 \pmod 8$ and all $1 \leq i \leq k$, we have $\Up_{i,k}\T = \Up_{i,k}$.
\item for $k \equiv 3 \pmod 8$ and all $1 \leq i \leq k$, we have $\Up_{i,k}\sT = -\Up_{i,k}$.
\item for $k \equiv 5 \pmod 8$ and all $1 \leq i \leq k$, we have $\Up_{i,k}\sT = \Up_{i,k}$.
\end{enumerate}
\end{lemma}

\begin{proof}
First suppose that $k \equiv -1 \pmod 8$. For the standard list of $k$-Clifford generators, we know from Lemma~\ref{Lemma:standardlist} that 
$\Gamma_{i,k}$ is symmetric exactly when $i$ is odd and is otherwise skew-symmetric.
The number of such odd integers is congruent to $0 \pmod 4$. Then by Part (1) of Lemma~\ref{transpose}, the construction of $\Up_{i,k}$ toggles the behavior in the odd cases
so that $\Up_{i,k}$ is skew-symmetric for all $i$.

Now if $k \equiv 1 \pmod 8$, then the construction of $\Up_{i,k}$ is achieved by toggling exactly the skew-symmetric generators, so that $\Up_{i,k}$ turn out to be symmetric for all $i$.

In the cases for $k \equiv 3 \pmod 8$ and $k \equiv 5 \pmod 8$, the same scheme is at play to ensure that $\Gamma_{i,k}$ are skew-$\sharp$-symmetric in the first case and $\sharp$-symmetric in the second case.
\end{proof}

\begin{lemma} \label{lemma:KO}
Suppose that $k$ is odd. 
In each of the four real cases, the unitaries defined in Construction~\ref{construction} satisfiy the right symmetries so that  $[Q\pr_{k-1}] \in KO_{k+1}(\fS^{k-1}) = \Z$ and $[U\pr_k] \in KO_{k+2}(\fS^k) = \Z$. 
\end{lemma}

In this situation, we are using the alternative version the unitary picture of $KO_3$ and $KO_7$ shown in the second line in the appropriate point of Table~\ref{unitaryTable} and justified by Propositions~\ref{prop-K(3)} and \ref{prop-K(-1)} in the Appendix.

\begin{proof}
First, it is easy to check that $Q\pr_{k-1}$ is a self-adjoint unitary and that $U\pr_k$ is a unitary in each case, from the fact that
$\Up_{i,k}$ are self-adjoint unitaries.

In the case $k \equiv -1 \pmod 8$, the Clifford generators $\Up_{1, k}, \dots, \Up_{k,k}$ are anti-symmetric. Hence
\begin{align*}
	(Q\pr_{k-1}) \Ttau(x)   
		&= \left( \sum_{i = 1}^k x_i \Up_{i,k} \right) \Ttau \\
		&= \sum_{i = 1}^{k} (-x_i) \Up_{i,k}\T\\
		&= \sum_{i = 1}^{k} (-x_i) ( - \Up_{i,k}) \\
		&= Q\pr_{k-1} (x) \;  \\
\text{and} \quad 
	(U\pr_k)\Ttau(x)   
		&= \left( \,  \sum_{i = 1}^{k} x_i \Up_{i,k}  + x_{k+1} i \, I_n \right) \Ttau \\
		&= \left( \,  \sum_{i = 1}^{k} (-x_i) (-\Up_{i,k})  + (-x_{k+1}) i \, I_n \right)  \\
		&= \left( \,  \sum_{i = 1}^{k} x_i \Up_{i,k} - x_{k+1} i \, I_n \right)  \\
		&= (U\pr_k)^*(x) \; .
\end{align*}
Therefore $[Q\pr_{k-1}] \in KO_0(\fS^{k-1})$ and $[U\pr_{k}] \in KO_1(\fS^{k})$.

Now, assume that $k \equiv 1 \pmod 8$. Then the Clifford generators $\Up_{1, k}, \dots, \Up_{k,k}$ are symmetric, so we have
\begin{align*}
	(Q\pr_{k-1})\Ttau(x)   
		&= \left( \sum_{i = 1}^k x_i \Up_{i,k} \right) \Ttau \\
		&= \sum_{i = 1}^{k} (-x_i) (  \Up_{i,k}) \\
		&= -Q\pr_{k-1} (x) \;  \\
\text{and} \quad 
	(U\pr_k)\Ttau(x)   
		&= \left( \,  \sum_{i = 1}^{k} x_i \Up_{i,k}  + x_{k+1} i \, I_n \right) \Ttau \\
		&= \left( \,  \sum_{i = 1}^{k} (-x_i) (\Up_{i,k})  + (-x_{k+1}) i \, I_n \right)  \\
		&= -U\pr_k(x) \; .
\end{align*}
It follows that $[Q\pr_{k-1}] \in KO_2(\fS^{k-1})$ and $[U\pr_{k}] \in KO_3(\fS^{k})$.

In the third case, $k \equiv 3 \pmod 8$, the elements $\Up_{1, k}, \dots, \Up_{k,k}$ satisfy $\Up_{i,k}\sT = - \Up_{i,k}$. Thus
\begin{align*}
	(Q\pr_{k-1}) \sTtau(x)   
		&= \left( \sum_{i = 1}^k x_i \Up_{i,k} \right) \sTtau \\
		&= \sum_{i = 1}^{k} (-x_i) \Up_{i,k}\sT\\
		&= \sum_{i = 1}^{k} (-x_i) ( - \Up_{i,k}) \\
		&= Q\pr_{k-1} (x) \;  \\
\text{and} \quad 
	(U\pr_k)\sTtau(x)   
		&= \left( \,  \sum_{i = 1}^{k} x_i \Up_{i,k}  + x_{k+1} i \, I_n \right) \sTtau \\
		&= \left( \,   \sum_{i = 1}^{k} x_i \Up_{i,k} - x_{k+1} i \, I_n \right)  \\
		&= (U_k\pr)^*(x) \; .
\end{align*}
Thus $[Q\pr_{k-1}] \in KO_4(\fS^{k-1})$ and $[U\pr_{k}] \in KO_5(\fS^{k})$.

Finally, if $k \equiv 5 \pmod 8$ we have $\Up_{i,k}\sT = \Up_{i,k}$ for all $i$ so
\begin{align*}
	(Q\pr_{k-1}) \sTtau(x)   
		&= \left( \sum_{i = 1}^k x_i \Up_{i,k} \right) \sTtau \\
		&= \sum_{i = 1}^{k} (-x_i) (  \Up_{i,k}) \\
		&= -Q\pr_{k-1} (x) \;  \\
\text{and} \quad 
	(U\pr_k)\sTtau(x)   
		&= \left( \,  \sum_{i = 1}^{k} x_i \Up_{i,k}  + x_{k+1} i \, I_n \right) \Ttau \\
		&= \left( \,  \sum_{i = 1}^{k} (-x_i) (\Up_{i,k})  + (-x_{k+1}) i \, I_n \right)  \\
		&= -U\pr_k(x) \; .
\end{align*}
It follows that $[Q\pr_{k-1}] \in KO_6(\fS^{k-1})$ and $[U\pr_{k}] \in KO_{7}(\fS^{k})$.
\end{proof}

\begin{thm} \label{main-thm}
In the real case, $[Q\pr_{k-1}]$ and $[U\pr_k]$ from Parts (1)-(4) of Construction~\ref{construction} represent generators of $KO_{k+1}(\fS^{k-1}) = \Z$ and $KO_{k+2}(\fS^k) = \Z$ respectively, for odd $k$. In other words,
$$g_{k-1} = [Q\pr_{k-1}] \qquad \text{and} \qquad g_k = [U\pr_k] \; .$$
\end{thm}

\begin{proof}
From Lemma~\ref{lemma:KO} we know that $[Q\pr_{k-1}]$ and $[U\pr_k]$ from Parts (1)-(4) of Construction~\ref{construction} represent elements of $KO_{k+1}(\fS^{k-1}) = \Z$ and $KO_{k+2}(\fS^k) = \Z$. It remains to show that the elements they represent are generators of the given groups.

The complexification map 
$$c_{k+1} \colon \widetilde{KO}_{k+1}(\fS^{k-1}) \rightarrow \widetilde{KU}_{k+1}(\fS^{k-1})$$ is defined by forgetting the extra symmetry that a self-adjoint unitary representing real $K$-theory must satisfy. In this case at hand, $c_{k+1}$ is known to be an isomorphism, because of the structure of the free $\mathcal{CR}$-module, namely that
$\widetilde{K}_*\crr(\fS^{k-1})$ is a free $\mathcal{CR}$-module with generator in $\widetilde{KO}_{k+1}(\fS^{k-1})$.
From Proposition~1 of \cite{SBS} we know that $[Q_{k-1}]$ represents a generator of 
$\widetilde{KU}_{k+1}(\fS^{k-1}) = \Z$. It follows that
$[Q\pr_{k-1}]$ must be a generator of $\widetilde{KO}_{k+1}(\fS^{k-1})$. 
The same argument shows that
$[U\pr_k]$ is a generator of $\widetilde{KO}_{k+2}(\fS^{k})$, because we know from Proposition~1 of \cite{SBS} that the same unitary represents a generator of $\widetilde{KU}_{k+2}(\fS^{k}) = \Z$
\end{proof}

\section{Appendix: An alternative unitary picture for $KO_{-1}$ and $KO_3$.} \label{newunitaries}

In Table~\ref{unitaryTable} describing the unitary picture of $KO_i(A, \tau)$, there are extra lines showing alternative unitary symmetries for $KO_{3}(A, \tau)$ and for $KO_{7}(A, \tau)$, which are different than those developed in \cite{BL}. In this section we will prove that the two unitary pictures of $KO$-theory in these cases are equivalent, so that either picture can be used. We also indicate how to translate from a unitary representing $KO_i(A, \tau)$ in one picture to a unitary representing $KO_i(A, \tau)$ in the other picture.

\begin{lemma}
Let $w = \left( \begin{smallmatrix} 0 & 1 \\ -1 & 0 \end{smallmatrix} \right)$. Then for any $x \in M_2(\C)$ we have 
\begin{enumerate}
\item $x^\sharp = w^* x^{\rm{Tr}} w$ 
\item $x^{\rm{Tr}} = w x^{\sharp} w^*$
\item $x^\sharp = w x^{\rm{Tr}} w^*$ 
\item $x^{\rm{Tr}} = w^* x^{\sharp} w$
\end{enumerate} \end{lemma}

\begin{proof}
The first statement follows from direct computation on $x = \left( \begin{smallmatrix} a & b \\ c & d \end{smallmatrix} \right)$. Statement (3) follows from Statement (1) and the fact that $w^2 = -1_2$, which is in the center of $M_2(\C)$, by
\begin{align*}
x^\sharp &= w^2 x^\sharp (w^*)^2 \\
	&= w^2 (w^* x^{\rm{Tr}} w) (w^*)^2 \\
	&= w x^{\rm{Tr}} w^*  \; .
\end{align*}
Statements (2) and (4) follow immediate from Statements (1) and (3) by conjugation.
\end{proof}

\begin{lemma} \label{lemma-sharp/trace}
Let $(A, \tau)$ be a $C \sp *$-algebra with antimultiplicative involution. Let $w = \left( \begin{smallmatrix} 0 & 1 \\ -1 & 0 \end{smallmatrix} \right)$.
Then a unitary $u \in M_2(A)$ satisfies
$u\Ttau =  u$ if and only if $v = u w$ satisfies $v\stau = - v$.
Furthermore, a unitary $u \in M_2(A)$ satisfies
$u\Ttau =  -u$ if and only if $v = u w$ satisfies $v\stau =  v$.
\end{lemma}

\begin{proof}
Suppose that $u$ satisfies $u\Ttau = u$. Then
\begin{align*}
v\stau &= (uw)\stau \\
	&= w^\sharp u\stau \\
	&= -w (w^* u\Ttau w) \\	
	&= -u\Ttau w \\
	&= -v \;.
\end{align*}
This proves the forward implication of the first statement. The other statements are proven by similarly.
\end{proof}

\begin{prop} \label{prop-K(-1)}
Let $A$ be a real unital $C \sp *$-algebra. Then $KO_{7}(A)$ is isomorphic to the group of homotopy classes of unitaries in 
$\bigcup_{n \in \N} M_{2n}(A_\C)$ that satisfy $v\stau = -v$, subject to the relation
$$[v] = \left[ \diag\left(v,\sm{0}{\1}{-\1}{0} \right) \right] \; .$$
\end{prop}

\begin{proof}
Let $A$ be given. Then we know from \cite{BL} (as indicated in Table~\ref{unitaryTable}) that $KO_{7}(A)$ is isomorphic to the group of homotopy classes of unitaries $u$ in 
$\bigcup_{n \in \N} M_{2n}(A_\C)$ that satisfy $u\Ttau = u$, subject to the relation
$$[u] = \left[ \diag\left(u, 1_2\right) \right] \; .$$
Now if $u$ is such a unitary, then $v = uw$ satisfies $v\stau = -v$ by Lemma~\ref{lemma-sharp/trace}. Furthermore, if $u = 1_2$, then we have $v = 1_2 \cdot w = \sm{0}{\1}{-\1}{0}$. Hence the $[u] \mapsto [uw]$ is a well-defined map from $KO_{7}(A)$ to the group described in the Proposition.

Conversely, if $v$ is a unitary that satisfies $v\stau = -v$, then $u = v(-w)$ satisfies $u \Ttau = u$, again by Lemma~\ref{lemma-sharp/trace}. So $[v] \mapsto [-vw]$ is a map from the group described in the Proposition to $KO_{7}(A)$.

Now $w(-w) = 1_2$, so it follows that these two group homomorphisms are mutual inverses.
\end{proof}

\begin{prop} \label{prop-K(3)}
Let $A$ be a real unital $C \sp *$-algebra. Then $KO_{3}(A)$ is isomorphic to the group of homotopy classes of unitaries in 
$\bigcup_{n \in \N} M_{2n}(A_\C)$ that satisfy $u\Ttau = -u$, subject to the relation
$$[u] = \left[ \diag\left(u,\sm{0}{\1}{-\1}{0} \right) \right] \; .$$
\end{prop}

\begin{proof}
The map $[u] \mapsto [uw]$ provides an isomorphism from the previous picture $KO_3(A)$ from \cite{BL} to the picture described in the Proposition. The proof is similar to the proof of Proposition~\ref{prop-K(-1)}, relying on Lemma~\ref{lemma-sharp/trace} again.
\end{proof}

\vspace{2cm}

\bibliographystyle{amsplain}
\bibliography{antipode-sphere}

\end{document}